\documentclass[12pt]{amsart}
\usepackage{amssymb,latexsym,amsmath}
\newtheorem{prop}{Proposition}

\newcommand{\Remark}{{\it Remark: }}
\newcommand{\Aside}{{\it Aside: }}
\newcommand{\EndAside}{{\it EndAside. }\smallskip}
\newcommand{\Claim}{{\it Claim: }}

\newcommand{\Diag}{\mathrm{Diag}}
\newcommand{\Tridiag}{\mathrm{Tridiag}}
\newcommand{\subdiag}{\mathrm{subdiag}}
\newcommand{\diag}{\mathrm{diag}}
\newcommand{\supdiag}{\mathrm{supdiag}}

\newcommand{\drew}{Drew, et al(2000)}
\newcommand{\elsner}{Elsner, et al(2003)}
\newcommand{\kkma}{Kailath, Kung and Morf(1979a)}
\newcommand{\kkmb}{Kailath, Kung and Morf(1979b)}
\newcommand{\ks}{Kailath and Sayed(1999)}
\newcommand{\rayleigh}{Rayleigh(1894)}
\newcommand{\trench}{Trench(2004)}
\newcommand{\weaver}{Weaver(1985)}

\begin{document}


\title
[Special Near-Toeplitz Matrices]
{The Eigen-Problem for Some Special \\
Near-Toeplitz Centro-Skew \\
Tridiagonal Matrices}



\date{\today}      
\maketitle

\centerline{by}

\centerline{Antonio Behn}
\centerline{Department of Mathematics}
\centerline{University of Chile}
\bigskip

\centerline{Kenneth R. Driessel}
\centerline{Mathematics Department}
\centerline{Iowa State University}
\bigskip

\centerline{Irvin R. Hentzel}
\centerline{Mathematics Department}
\centerline{Iowa State University}

\begin{abstract}
Let $n\ge 2$ be an integer. Let $R_n$ denote the $n\times n$ tridiagonal matrix with $-1$'s on the sub-diagonal, $1$'s on the super-diagonal, $-1$ in the $(1,1)$ entry, $1$ in the $(n,n)$ entry and zeros elsewhere. We find the eigen-pairs of the matrices $R_n$.
\end{abstract}  

\dedicatory{}


\section*{Table of contents}

\begin{itemize}
\item
Introduction
\item
A Reduction
\item
Tridiagonal Toeplitz Matrices
\item
The Eigenvalues
\item
Acknowledgements
\item
References
\end{itemize}

\newpage

\section*{Introduction} 

We consider some special n-by-n, near-Toeplitz, tridiagonal matrices with entries from the set $\{0,1,-1\}$. In particular, we consider tridiagonal matrices having the following form:
$$
R_n := \Tridiag (\subdiag,\diag,\supdiag)
$$
where
\begin{itemize}
\item
$\subdiag:=(-1,-1,\dots,-1)$,
\item
$\diag:=(-1,0,0,\dots,0,0,1)$, and
\item
$\supdiag:=(1,1,\dots,1,1)$.
\end{itemize}
In other words, $R_n$ is the tridiagonal matrix that 
has all $-1$'s on the subdiagonal, all 0's on the diagonal except for a $-1$ in the (1,1) entry and a 1 in 
the (n,n) entry, and has all 1's on the superdiagonal. For example, when
$n:=4$ we have:
$$
R_4:=
\begin{pmatrix}
-1 &1  &0  &0  \\
-1 &0  &1  &0  \\
0  &-1 &0  &1  \\
0  &0  &-1 &1  \\

\end{pmatrix}.
$$
We find the eigenvalue-vector pairs (or {\bf eigen-pairs} for short) of the matrices $R_n$. In particular, we prove the following result:

\begin{prop}
The eigenvalues of the matrix $R_n$ are 0 and 
$2i\cos(j\theta)$
for $j=1,\dots,n-1$ 
where $\theta:=\pi/n$. 
\end{prop}

\Remark If $n$ is even then, 0 is an eigenvalue with multiplicity 2.

Let $\mathcal{T}_n$ denote the set of $n\times n$ tridiagonal real matrices that satisfy the following conditions: the sub-diagonal is negative, the super-diagonal is positive, the $(1,1)$ entry is negative, the $(n,n)$ entry is positive and all other entries are zero. \drew\  conjectured that this sign pattern class contains matrices with arbitrary spectra. Note that the matrix $R_n$ is in this class. They provided evidence for the conjecture. (\elsner\ provided further evidence.)
We believe that understanding the properties of $R_n$ may be an important step toward understanding the arbitrary spectrum conjecture for the sign pattern class $\mathcal{T}_n$.   

Here is a summary of the contents. In the section with title ``A Reduction'', we show that $R_n$ is similar to a near skew-symmetric, tridiagonal, Toeplitz matrix. In the section with title, ``Tridiagonal Toeplitz Matrices'' we review the solution of the eigen-problem for such matrices; in particular, we determine the eigen-pairs for a skew-symmetric tridiagonal Toeplitz matrix. In the section with title ``The Eigenvalues'', we determine the eigenvalues of $R_n$; in other words, we prove the proposition given above. 

\Aside Let $E$ be the $n\times n$ matrix defined by 
$E(i,j):= \delta(i+j,n+1)$, for $1\leq i,j \leq n$, where $\delta$ is the Kronecker delta. This matrix is called the {\bf exchange matrix} or the {\bf flip matrix}. 

Let $P$ be an $n\times n$ matrix. Then $P$ is a {\bf centro-symmetric matrix} if $EPE=P$ and $P$ is a {\bf centro-skew matrix} if $EPE=-P$. 

Here is an older reference on centro-symmetry: \weaver. Here is a more recent reference: \trench. 
These papers contain further references. We shall not use the theory of centro-symmetric matrices in this paper. 
\EndAside


\section*{A reduction} 

Recall that we are considering special matrices $R_n$ where $R_n$ is the tridiagonal matrix which 
that has all $-1$'s on the subdiagonal, has all 0's on the diagonal except for a $-1$ in the (1,1) entry and a 1 in the (n,n) entry, and has all 1's on the superdiagonal. We want to find the eigenvalues and eigenvectors of these matrices.

Here is another description of $R_n$. Let $Z_n$ denote the {\bf lower shift matrix} which is the matrix that has 1's on the subdiagonal and 0's elsewhere. For example, when $n=4$, we have
$$
Z_4 := 
\begin{pmatrix}
0 & 0 & 0 & 0  \\
1 & 0 & 0 & 0  \\
0 & 1 & 0 & 0  \\
0 & 0 & 1 & 0 
\end{pmatrix}.
$$

\Aside
Here we follow notation and terminology used by T. Kailath. See, for example, \kkma, \kkmb\  or \ks.
\EndAside

Note that 
$$
R_n = Z_n^T - Z_n - e_1 e_1^T + e_n e_n^T
$$
where $e_k$ denotes the column vector which has 1 in the $k$th coordinate and 0's elsewhere. 

\begin{prop}{\bf Reduction.}
Let $S_n:=I_n+Z_n$ where $I_n$ is the $n\times n$ identity matrix. Then
$$
S_n^{-1} R_n S_n = K_n + e_n e_{n-1}^T
$$
where $K_n:= Z_n^T - Z_n$.
\end{prop}

\Remark Here is a picture of $S_4$:
$$
S_4:=
\begin{pmatrix}
1 &0 &0 &0 \\
1 &1 &0 &0 \\
0 &1 &1 &0 \\
0 &0 &1 &1
\end{pmatrix}
$$

Note that $I_n+Z_n$ is invertible since $Z_n$ is nilpotent. In fact
$$
(I_n+Z_n)^{-1} = I_n - Z_n + Z_n^2 - \cdots \pm Z_n^{n-1}.
$$

When $n=4$, for example, the conclusion of the proposition is
\begin{align*}
\begin{pmatrix}
 1 &  0 &  0 & 0 \\
-1 &  1 &  0 & 0 \\
 1 & -1 &  1 & 0 \\
-1 &  1 & -1 & 1 \\
\end{pmatrix}
&\begin{pmatrix}
-1 &  1 &  0 & 0  \\
-1 &  0 &  1 & 0  \\
0  & -1 &  0 & 1  \\
0  &  0 & -1 & 1  \\
\end{pmatrix}
\begin{pmatrix}
1 & 0 & 0 & 0  \\
1 & 1 & 0 & 0  \\
0 & 1 & 1 & 0  \\
0 & 0 & 1 & 1  \\
\end{pmatrix} \\
&=
\begin{pmatrix}
 0 &  1 & 0 & 0 \\
-1 &  0 & 1 & 0  \\
 0 & -1 & 0 & 1  \\
 0 &  0 & 0 & 0  \\
\end{pmatrix}.
\end{align*}

\begin{proof}
In order to reduce the number of symbols in this proof, we omit the subscript $n$ on $I$, $Z$, $R$ and $S$.

Note that the equation in the conclusion of the proposition is equivalent to 
\begin{equation}
RS = SK + Se_n e_{n-1}^T \tag{$*$}.
\end{equation}

Also note
\begin{align*}
RS &= (K - e_1 e_1^T + e_n e_n^T) S \\
&= KS + (e_n e_n^T - e_1 e_1^T)S \\
&= SK + [K,S] + (e_n e_n^T - e_1 e_1^T)S
\end{align*}
where $[X,Y]:= XY - YX$.

Hence, to prove $(*)$, we need only prove the following

\Claim
$ [K,S] = Se_n e_{n-1}^T + (e_1 e_1^T - e_n e_n^T)S$

We have 
\begin{align*}
[K,S] &= [K,I+Z] = [K,Z] = [Z^T-Z,Z] = [Z^T,Z] \\
&= Z^T Z - ZZ^T = \Diag(1,1,\dots,1,0) - \Diag(0,1,\dots,1,1) \\
&= e_1 e_1^T - e_n e_n^T.
\end{align*}

We also have
\begin{align*}
Se_n e_{n-1}^T &+ (e_1 e_1^T - e_n e_n^T)S \\
&= 
(I+Z)e_n e_{n-1}^T + (e_1 e_1^T - e_n e_n^T)(I+Z) \\
&=
e_1 e_1^T - e_n e_n^T.
\end{align*}
\end{proof}


\section*{Tridiagonal Toeplitz matrices} 
In this section, we consider the skew symmetric tridiagonal matrices
$K_n := Z_n^T - Z_n$.
We determine the eigenvalues and eigenvectors of these matrices.  

Let $T_n(a,b,c,):=\Tridiag(\subdiag, \diag, \supdiag)$
denote the $n\times n$ tridiagonal matrix 
determined by
$\subdiag:=(a,a,\dots,a,a)$,
$\diag:=(b,b,\dots,b,b)$, and
$\supdiag:=(c,c,\dots,c,c)$.
In other words, $T_n(a,b,c)$ is the tridiagonal matrix that has all $a$'s on the subdiagonal, has all $b$'s on the diagonal and has all $c$'s on the superdiagonal. Recall that matrices with constant diagonals are called {\bf Toeplitz} matrices. For example, when
$n:=4$ we have:
$$
T_4(a,b,c) :=
\begin{pmatrix}
b &c  &0  &0  \\
a &b  &c  &0  \\
0 &a  &b  &c  \\
0 &0  &a  &b  \\
\end{pmatrix}.
$$
We want to find the eigenvalues and eigenvectors of the matrices $T_n(a,b,c)$. 

The following result is well-known. 

\begin{prop}
The matrix $T_n(a,b,c)$ is diagonally similar to the symmetric matrix $T_n(\sqrt{ac},b,\sqrt{ac})$ provided $ac \neq 0$. In particular, 
$D T_n(a,b,c) D^{-1}= T_n(\sqrt{ac},b,\sqrt{ac})$
where $D:= \Diag(1,d,d^2,\dots,d^n)$ and
$d:=\sqrt{c/a}$. 
Furthermore,  $T_n(a,b,c)u=\lambda u$ iff
$T_n(\sqrt{ac},b,\sqrt{ac})Du = \lambda Du$ and hence the eigen-pairs of $T_n(a,b,c)$ are determined by the eigen-pairs of $T_n(\sqrt{ac},b,\sqrt{ac})$.
\end{prop}

We call the transformation 
$$
T_n(a,b,c)\mapsto D T_n(a,b,c) D^{-1}=
T_n(\sqrt{ac},b,\sqrt{ac})
$$
the {\bf diagonal similarity symmetrizing transformation}.

Here is an example:
\begin{align*}
&\begin{pmatrix}
 1 & 0 & 0 \\
 0 & d & 0 \\
 0 & 0 & d^2 \\
\end{pmatrix}
\begin{pmatrix}
b & c & 0   \\
a & b & c   \\
0 & a & b  \\
\end{pmatrix}
\begin{pmatrix}
1 & 0      & 0   \\
0 & d^{-1} & 0   \\
0 & 0      & d^{-2} \\
\end{pmatrix} \\
&=
\begin{pmatrix}
 b &  c/d & 0   \\
da &  b   & c/d \\
 0 & da   & b   \\
\end{pmatrix}
=
\begin{pmatrix}
 b        &  \sqrt{ac} & 0         \\
\sqrt{ac} &  b         & \sqrt{ac} \\
 0        & \sqrt{ac}  & b         \\
\end{pmatrix}.
\end{align*}

The following result is well-known. (See, for example, \rayleigh.) We include a simple proof for the reader's convenience. 

\begin{prop} {\bf Symmetric tridiagonal Toeplitz eigen-pairs.} 
The eigenvalues $\lambda_j$ and corresponding eigenvectors $u_j$ of the $n\times n$ matrix
$T_n(a,b,a)$ are, for $j=1,\dots,n$, given by
$\lambda_j := b + 2a \cos(j\theta)$ and
$u_j := (\sin(j\theta), \sin(2j\theta), \cdots,
\sin(nj\theta))^T $, where $\theta:=\pi/(n+1)$.
\end{prop}

\begin{proof}
Recall that matrices $A$ and $bI+A$ have the same eigenvectors. Also recall that $\lambda$ is an eigenvalue of $A$ iff $b+\lambda$ is an eigenvalue of $bI+A$. 

Recall that, for $a\neq0$, the matrices $A$ and $aA$ have the same eigenvectors. Also recall that $\lambda$ is an eigenvalue of $A$ iff $a\lambda$ is an eigenvlue of $aA$.

Using these reductions, we see that we only need to verify the eigen-pairs of $T_n(1,0,1)$. To reduce the number of symbols in this calculation we take $n=3$.
In this case, we have $\theta:=\pi/4$ and
\begin{align*}
T_3(1,0,1)u_j &=
\begin{pmatrix}
0 & 1 & 0 \\
1 & 0 & 1 \\
0 & 1 & 0 \\
\end{pmatrix}
\begin{pmatrix}
 \sin(j\theta) \\
 \sin(2j\theta)\\
 \sin(3j\theta) \\
\end{pmatrix}
=
\begin{pmatrix}
\sin(0j\theta)+\sin(2j\theta) \\
\sin(1j\theta)+\sin(3j\theta) \\
\sin(2j\theta)+\sin(4j\theta) \\
\end{pmatrix} \\
&= 2\cos(j\theta)
\begin{pmatrix}
\sin(j\theta)  \\
\sin(2j\theta)  \\
\sin(3j\theta)  \\
\end{pmatrix}.
\end{align*} 
The last equality follows from the well-known 3 term recurrence relation for the sine function:
$$
\sin(k+1)\alpha - 2\cos\alpha\sin(k\alpha)
+\sin(k-1)\alpha=0.
$$
\end{proof}

\begin{prop} {\bf Skew-symmetric tridiagonal Toeplitz eigen-pairs.} 
The eigenvalues $\lambda_j$ and corresponding eigenvectors $u_j$ of the $n\times n$ matrix
$K_n:=Z_n^T-Z_n$ are, for $j=1,\dots,n$, given by
$\lambda_j := 2i \cos(j\theta)$ and
$
u_j:=(i\sin(j\theta),
i^2\sin(2j\theta),
i^3\sin(3j\theta),
\cdots,
i^n\sin(nj\theta))
$, where $\theta:=\pi/(n+1)$.
\end{prop}

\begin{proof}
Note that $K_n=Z_n^T-Z_n=T_n(-1,0,1)$. We apply the diagonal similarity symmetrizing transformation with $D^{-1}:=\Diag(1,i,i^2,\dots,i^n)$ to get
$$ 
DT_n(-1,0,1)D^{-1}= T_n(i,0,i) = iT_n(1,0,1).
$$
We know the eigen-pairs of $T_n(1,0,1)$ from above. Thus we can easily determine the eigen-pairs of $K_n$. In particular, if $T_n(1,0,1)u = \lambda u$ then
$D T_n(-1,0,1) D^{-1} u = i T_n(1,0,1)
= i\lambda u$. Hence
$T_n(-1,0,-1)(D^{-1}u) = i\lambda D^{-1}u$.
\end{proof}


\section*{The Eigenvalues} 

\begin{prop}
The eigenvalues of the matrix $R_n$ are 0 and 
$2i\cos(j\theta)$
for $j=1,\dots,n-1$ 
where $\theta:=\pi/n$. 
\end{prop}

\begin{proof}
It is clear that the vector $(1,1,\dots,1)^T$ is an eigenvector with 0 as corresponding eigenvalue. 

From the reduction proposition, we have
$$
S_n^{-1}R_n S_n =
K_n+e_n e_{n-1}^T.
$$
Note that
$$
K_n+e_n e_{n-1}^T =
\begin{pmatrix}
K_{n-1} & e_{n-1}\\
0       & 0
\end{pmatrix}.
$$
Also note that if $K_{n-1}u=\lambda u$ then
$$
\begin{pmatrix}
K_{n-1} & e_{n-1}\\
0       & 0
\end{pmatrix}
\begin{pmatrix}
u \\
0
\end{pmatrix}
= K_{n-1}u = \lambda u = \lambda
\begin{pmatrix}
u \\
0
\end{pmatrix}.
$$ 
\end{proof}

\Remark The eigenvectors of $R_n$ are determined in the proof of the proposition. 


\section*{Acknowledgements} 

This paper was written in October of 2010 while Behn was visiting the mathematics department at Iowa State university (under grant number FONDECYT 1100135). 

We thank Wayne Barrett (Brigham Young University) for his careful reading of this paper and his constructive comments about it. 

Driessel thanks Wolfgang Kliemann, chair of the Mathematics Department at Iowa State, for arranging his affiliation with that department. 


\section*{References}

\begin{itemize}

\item
Drew, J.H.; Johnson, C.R.; Olesky, D.D.; and van den Driesche, P. (2000)
Spectrally arbitrary patterns, 
\emph{Linear Algebra and Appl.}
308, 121-137
\smallskip

\item
Elsner, L.; Olesky, D.D.; and van den Driesche, P. (2003)
Low rank perturbations and the spectrum of a tri-diagonal sign pattern,
\emph{Linear Algebra and Appl.}
308, 121-137
\smallskip

\item
Kailath, T.; Kung, S.Y.; and Morf, M. (1979a)
Displacement ranks of matrices and linear equations,
\emph{J. Math. Anal. Appl.}
68, 395-407
\smallskip

\item
Kailath, T.; Kung, S.Y.; and Morf, M. (1979b)
Displacement ranks of a matrix,
\emph{Bull. Amer. Math. Soc.}
1, 769-773
\smallskip

\item
Kailath, T. and Sayed, A.H. (1999)
\emph{Fast Reliable Algorithms for Matrices with Structure},
SIAM
\smallskip

\item
Rayleigh, J.W.S. (1894) 
\emph {The Theory of Sound}, Macmillan (Reprinted by Dover in 1945.)
\smallskip

\item
Trench, W.F. (2004)
Characterization and properties of matrices with generalized symmetry or skew symmetry,
\emph{Lin. Alg. Appl.}
377, 207-218
\smallskip

\item
Weaver, J.R. (1985)
Centrosymmetric (Cross-Symmetric) Matrices,
\emph{American Mathematical Monthly}
92, 711-717
\smallskip

\end{itemize}

\end{document}